\documentclass[12pt]{amsart}
\newfont{\sheaf}{eusm10 scaled\magstep1}

\usepackage{graphics}
\usepackage{color}

\newcommand{\ra}{\ensuremath{\rightarrow}}

\def\eea{\end{eqnarray*}}
\def\bea{\begin{eqnarray*}}
\def\Bbb{\bf}

\def\QQ{{\Bbb Q}}

\def\s{\sigma}
\def\la{\lambda}
\def\ga{\gamma}
\def\Ga{\Gamma}

\newcommand{\Proof}{{\it Proof. }}

\newtheorem{teo}{Theorem}[section]
\newtheorem{df}[teo]{Definition}
\newtheorem{lem}[teo]{Lemma}
\newtheorem{cor}[teo]{Corollary}

\newtheorem{oss}[teo]{Remark}
\newtheorem{prop}[teo]{Proposition}
\newtheorem{conj}[teo]{Conjecture}
\newtheorem{Question}[teo]{Question}

\newcommand{\C}{\ensuremath{\mathbb{C}}}
\newcommand{\R}{\ensuremath{\mathbb{R}}}
\newcommand{\Z}{\ensuremath{\mathbb{Z}}}
\newcommand{\Q}{\ensuremath{\mathbb{Q}}}
\newcommand{\F}{\ensuremath{\mathbb{F}}}

\newcommand{\N}{\ensuremath{\mathbb{N}}}
\newcommand{\hol}{\ensuremath{\mathcal{O}}}

\newcommand{\PP}{\ensuremath{\mathbb{P}}}

\newcommand{\FS}{\ensuremath{\mathfrak{S}}}

\newcommand{\gc}{\ensuremath{\mathfrak{c}}}

\DeclareMathOperator{\Aut}{Aut}
\DeclareMathOperator{\Ad}{Ad}
\DeclareMathOperator{\Out}{Out}
\DeclareMathOperator{\Inn}{Inn}
\DeclareMathOperator{\Gal}{Gal}
\DeclareMathOperator{\Map}{Map}
\DeclareMathOperator{\Epi}{Epi}

\title[Absolute Galois acts faithfully on  moduli] {Faithful actions of the absolute Galois group  on connected components of 
moduli spaces}

\author{Ingrid Bauer, Fabrizio Catanese and Fritz Grunewald}

\begin{document}

\date{\today}

\begin{abstract} 
We give  a canonical procedure associating to an algebraic number $a$  first a hyperelliptic curve $C_a$,  
and then a triangle curve $(D_a, G_a)$ obtained through the normal closure of an  associated Belyi function.

In this way we show  that the absolute Galois group  $\Gal(\bar{\Q} /\Q)$ acts faithfully 
 on the set of isomorphism classes of marked
triangle curves, and on the set of connected
components of  marked moduli spaces of surfaces isogenous to a higher product
(these are the free quotients 
of a product  $C_1 \times C_2$ of curves of respective genera $g_1, g_2 \geq 2$ by the action of a finite group $G$).
We  show then, using again the surfaces isogenous to a product, first that it acts faithfully 
on the set of connected
components of   moduli spaces of surfaces of general type (amending an incorrect proof in a previous ArXiv version of the paper);
and then, as a consequence, we obtain that  for every  element $\sigma \in \Gal(\bar{\Q} /\Q)$, not in the conjugacy class of
complex conjugation, there exists a surface of general type $X$ such that
$X$ and the Galois conjugate surface $X^{\sigma}$  have nonisomorphic fundamental groups.

Using polynomials with only two critical values, we can moreover exhibit infinitely many explicit examples
of such a situation. 
  \end{abstract}
\maketitle

\section*{Introduction}
In the 60's J. P. Serre showed  in \cite{serre}  that there exists a field automorphism  $\sigma \in
\Gal(\bar{\Q} /\Q)
$, and a variety $X$ defined over $\bar{\Q}$ such that 
$X$ and the Galois conjugate variety $X^{\sigma}$  have non isomorphic fundamental groups, in
particular they are not homeomorphic.

In  this note we give new examples  of this phenomenon, using  the so-called `surfaces isogenous to a product' whose weak rigidity
was proven in \cite{isogenous} (see also \cite{cat03}) and which by definition are  quotients of a product of curves ($C_1
\times C_2$) of respective genera at least $2$ by the free action of  a finite group $G$.

One of our main results is a  strong
sharpening of the phenomenon discovered by Serre: observe in this respect that, if   $\mathfrak c$
denotes complex conjugation, then $X$ and $X^{\gc}$ are diffeomorphic. 

\begin{teo} \label{fundamentalgroup}  If $\sigma \in \Gal(\bar{\Q} /\Q)$ is not in the conjugacy class of $\gc$,
then there exists a surface isogenous to a product $X$ such that $X$ and the Galois conjugate variety $X^{\sigma}$  have non isomorphic fundamental groups. 
\end{teo}

Moreover, we give some faithful actions of the  
absolute  Galois group  $\Gal(\bar{\Q} /\Q)$, related among them. 

The following results are based on the concept of a (symmetry-) marked variety. A marked variety is  a
   triple  $(X,G, \eta)$ where $X$ is a projective variety,  and $ \eta \colon G \ra \Aut(X)$ is an injective homomorphism
   (one says  also  that we have an effective action of the group $G$ on $X$): 
   here two such triples  $(X,G,\eta)$, $(X',G', \eta')$ are  isomorphic iff there are isomorphisms $ f \colon  X \ra X'$,
 and $\psi \colon  G \ra G'$ such that $f$ carries  the first action $\eta$ to the second one $\eta'$ (i.e., such that $ \eta ' \circ \psi = \Ad(f) \circ \eta$,
 where $ \Ad (f) (\phi) : = f \phi f^{-1}$). A particular case of marking is the one where $ G \subset \Aut(X)$ and $\eta$ is the inclusion:
 in this case we may denote a marked variety simply by the pair $(X,G)$.

\begin{teo}\label{markedtriangleintro}To any algebraic number $a \notin \Z$ there corresponds, through a canonical procedure (depending on an integer $g \geq 3$), a marked triangle curve $(D_a, G_a)$.

This correspondence yields a faithful action of the  absolute Galois group $\Gal(\bar{\Q} /\Q)$ on the  set of isomorphism classes
of marked triangle curves.
\end{teo}

\begin{teo} \label{marked isogenousintro}The absolute Galois group $\Gal(\bar{\Q} /\Q)$ acts faithfully on the  set of connected
components of the (coarse) moduli spaces of \'etale marked  surfaces of general type. 
\end{teo}

With a rather  elaborate strategy we can then show the stronger result:

\begin{teo} \label{isogenous} The absolute Galois group $\Gal(\bar{\Q} /\Q)$ acts faithfully on the  set of connected
components of the (coarse) moduli space of    surfaces of general type. 
\end{teo}

Our method is closely related to the so-called theory of  `dessins d' enfants' (see \cite{groth2}).
Dessins d' enfants are,  in view of Riemann's existence theorem (generalized by Grauert and Remmert in
\cite{g-r}),  a combinatorial  way to look at the monodromies of algebraic functions with only three branch
points. We emphasize once more how we make here an essential use of Belyi functions (\cite{belyi}) and of
their functoriality. Our point of view is  however more related to the normal closure of Belyi functions,
the so called marked triangle curves, i.e., pairs $(C,G)$ with $ G \subset \Aut (C)$ such that
the quotient $ C/G \cong \PP^1$ and the quotient map is branched exactly in three points.

In the first section we describe a simple but canonical construction which, 
for each choice of an integer $ g 
\geq 3$, associates to a complex number $ a \in \C \setminus \Q$
a hyperelliptic curve $C_a$ of genus $g$, and  in such a way that $C_a \cong C_b$ iff $ a = b$.

In the later sections we construct the associated triangle curves $(D_a, G_a)$ and prove the   above theorems.

It would be interesting to obtain similar types of results, for instance forgetting about the markings in the case of triangle curves, 
or even  using  only   Beauville surfaces (these are the surfaces isogenous to a product which are rigid: see \cite{isogenous}
 for the definition of
Beauville surfaces and \cite{cat03},\cite{bcg}, \cite{almeria} for further properties of these).

Theorem \ref{isogenous} was    announced by the second author at the Alghero Conference 'Topology of algebraic varieties' in
september 2006, and asserted with an incorrect proof  in the previous ArXiv version of the paper (\cite{AbsoluteGalois}).
The survey article \cite{catbumi} then transformed  some of the  theorems of \cite{AbsoluteGalois} into conjectures.
The present article then takes up again some conjectures made previously in \cite{AbsoluteGalois} and repeated
in \cite{catbumi}.

The main new input  of the present paper is the systematic use of twists of the second component of an action on a product $C_1 \times C_2$
by an automorphism of the group $G$ and the discovery that this leads to an injective  homomorphism of the Kernel $\frak K$ of the action (on the set of
connected components $\pi_0 ( \frak M)$ of the moduli space of surfaces of general type) into some Abelian group
of the form $ \oplus_G ( Z (\Out (G))$, $Z$ denoting the centre of a group. Then we use a known result (cf. \cite{F-J})
that $\Gal(\bar{\Q} /\Q)$ does not contain any nontrivial normal abelian subgroup.

Observe that
Robert Easton and Ravi Vakil (\cite{e-v}), with a completely different type of examples,
showed that the Galois group  $\Gal(\bar{\Q} /\Q)$ operates faithfully on the set of irreducible components of the
moduli spaces of surfaces of general type.

In the last section we use Beauville surfaces and polynomials with two critical values in order
to produce   infinitely many explicit and simple examples of  
pairs of surfaces of general type with nonisomorphic fundamental groups
which are conjugate under the absolute Galois group (observe in particular that  the two 
fundamental groups have then  isomorphic
profinite completions).

\section{Very special hyperelliptic curves}\label{1}

  Fix a positive integer $g \in \N$, $g \geq 3$, and define,
 for any complex number $ a \in \C \setminus \{ -2g,0,1, \ldots , 2g-1 \} $, $C_a$  as the hyperelliptic
curve   of genus $g$ 
$$ w^2 = (z-a) (z + 2g) \Pi_{i=0}^{2g-1} (z-i) $$ branched over 
$\{ -2g, 0,1, \ldots , 2g-1, a \} \subset \PP^1_{\C}$.

\begin{prop}\label{a=b} \
\begin{enumerate}
\item  Consider two complex numbers  $a,b$ such that $a \in  \C 
\setminus \Q $: then $C_a
\cong C_b$ if and only if $a = b$.
\item Assume now  that $g \geq 6$ and let $a,b \in  \C
\setminus
\{ -2g, 0,1, \ldots , 2g-1 \} $ be two complex numbers. Then $C_a
\cong C_b$ if and only if $a = b$.
\end{enumerate}
\end{prop}

\begin{proof}

One direction being obvious,  assume that $C_a \cong C_b$.

1)
Then the two sets with $2g+2$ elements
$B_a : = \{ -2g, 0,1, \ldots , 2g-1, a \}$ and
$B_b : = \{ -2g, 0,1, \ldots , 2g-1, b \}$ are projectively 
equivalent over $\mathbb{C}$ (the latter  set $B_b$
has also cardinality $2g+2$ since $C_a \cong C_b$  and $C_a$ smooth 
implies that also
$C_b$ is smooth).

In fact, this projectivity $\varphi$ is defined over $\Q$, since 
there are three rational numbers which are
carried into three rational numbers (because $g \geq 2$).

Since $ a \notin \Q$ it follows that $\varphi(a)  \notin \Q$ hence $\varphi(a) =  b \notin \Q$ and $\varphi$ 
maps $ B: = \{-2g, 0,1,\ldots 2g-1\}$ to
$ B$, and in particular  $\varphi$ has finite order.  Since $\varphi$ yields an automorphism of $\PP^1_{\R}$, it 
either leaves the cyclical order of
$(-2g, 0,1, \ldots , 2g-1)$ invariant or reverses it,  and since $g \geq 3$ 
we see that there are   $3$ consecutive
integers such that $\varphi$ maps them to $3$ consecutive integers. 
Therefore $\varphi$ is either an
integer translation,  or an affine symmetry  of the form $ x \mapsto 
- x + 2n$, where $ 2 n \in \Z$. In the former case   $\varphi
= id$, since it has finite order, and it follows  in particular that $a = b$. In the 
latter case it must be $2g + 2n = \varphi
(-2g) = 2g-1$ and $2n = \varphi (0) = 2g-2$,
  and we derive the contradiction $ -1 = 2n = 2g-2$.

2) The case where $a , b \in \Q$ is similar to the previous one: 
$\varphi$ preserves or reverses  the cyclical order of the
two sets, and we are done as before  if $\varphi (a) = b$.

Observe that the set $ B_a : = \{-2g , 0,1,\ldots, 2g-1, a\}$ admits 
a parabolic transformation $\psi (x) : = x
+ 1$, with  $\infty$ as fixed point, with the  properties that $ | 
\psi (B_a ) \cap B_a | \geq 2g-1$, and
that there is an element $0 \in B_a$ such that $\psi (0),  \dots , 
\psi^{2g-1} (0) \in B_a$.
If $B_b$ is projectively
equivalent to $B_a$, then also $B_b$ inherits such a parabolic 
transformation $
\tau$ with this property. 

Assume that  the fixed point of 
$\tau$ is $\infty$:  then $\varphi (\infty) = \infty$, $\varphi $ is affine and,
since $ \varphi \circ \psi = \tau \circ \varphi  $, $\tau$ is of the form $\tau (x) = x + m$, for some $ m \in \Q$,
and moreover $\varphi (x) = m x + r$, for some  $ r \in \Q$.
In fact, if we set $\varphi (x) = a x + r$,  $ \varphi \circ \psi = \tau \circ \varphi \Leftrightarrow a (x+1) + r = ax + r + m  \Leftrightarrow m=a.$

Since $\tau (x) = x + m$,  the above property implies that
$ m = \pm 1$, hence also  $\varphi (x) = \pm  x + r$.

\noindent
{\bf Claim:} $\varphi (x) = \pm  x + r \ \implies \ a=b$.

\noindent
{\em First proof of the claim:}

 If $\varphi (x) =  x + r$, then 
$B_b$ contains $\{r, \dots, (2g-1)+ r\} $ and either 

\begin{enumerate}
\item
$r=0$ and  $\varphi = id$ or 
\item
$r=1$, $B_b$ contains $  2g$ and $-2g+1$, a contradiction, or
\item
$r=- 1$, $B_b$ contains $  -1$ and $-2g-1$, a contradiction. 

\end{enumerate}
Similarly, if  
$\varphi (x) = - x + r$, $B_b$ contains $\{r - (2g-1), \dots,  r\} $ 
 and either 

\begin{enumerate}
\item
$r=2g-1 $ and  $a=b = 4g-1$ or 
\item
$r=2g$, $B_b$ contains $  2g$ and $4g$, a contradiction, or
\item
$r= 2g-2$, $B_b$ contains $  -1$ and $-4g-2$, a contradiction. 

\end{enumerate}

\qed

Let $w$ be the fixed point of $\tau$: then we may assume that $ w \in \Q$ and we must 
exclude this case. Observe that in the
set $B_b$ each consecutive triple of points is a triple of 
consecutive integers, if no element in the
triple is $ - 2g$ or $b$. This excludes at most six triples. Keep  in mind 
that $ a \in B_a$ and consider all the
consecutive triples of integers in the set $\{0,1,2,3,4,5,6,7, 8 , 9, 
10, 11\}$: at most two such triples are not
a consecutive triple of points of $B_a$. We conclude that there is 
a triple of consecutive integers in the set
$\{0,1,2,3,4,5,6,7, 8 , 9, 10, 11\}$ mapping to a triple of 
consecutive integers under
$\varphi$. Then either $\varphi$ is an integer translation $ x 
\mapsto x + n$, or it is a symmetry  $ x
\mapsto - x + 2 n$ with $ 2n \in \Z$.

\noindent
{\em Second proof of the claim:}

In both cases the intervals equal to the respective convex 
spans of the sets $B_a$, $B_b$ are sent to
each other by $\varphi$, in particular the  length is preserved and 
the extremal points are permuted. If  $ a
\in [-2g, 2g-1]$ also  $ b \in [-2g, 2g-1]$ and in the translation 
case $n=0$, so that $\varphi (x) = x$ and $a
= b$. We see right away that $\varphi$  cannot be a symmetry, because 
only two points belong to the left
half  of the interval.

 If  $ a  < -2g $ the interval has length $ 
2g-1 -a$, if $ a > 2g-1$ the interval has length $
2g   + a$. Hence, if both  $ a , b < -2g $, since the length is preserved, we find that $a=b$;
 similarly if $ a, b > 2g-1$. 

By symmetry of the situation, we only need to exclude the case $ a  < -2g $, $ b >  2g-1$: here we must have $ 2g-1 
-a = 2 g + b$, i.e., $ a = -b-1$. We have already treated the case where
$b = \varphi (a)$, hence we have a translation $\varphi (x ) = x + n$ and  since $\varphi (2g-1 ) = b$, $\varphi 
(2g-2 ) = 2g-1$, it follows that  $n=1$, and then
$\varphi (-2g ) = -2 g + 1$ gives a contradiction.
\qed

\end{proof}

We shall assume from now on that $a,b \in \bar{\Q} \setminus \Q$ and that there is a field  automorphism
$\sigma \in \Gal(\bar{\Q} /\Q)$ such that $\sigma (a) = b$.  (Obviously, for any  $\sigma \in \Gal(\bar{\Q}
/\Q)$ different from the identity, there are  $a,b \in \bar{\Q} \setminus \Q$ with  $\sigma (a) =
b$ and $ a \neq b$.)

The following is a special case of Belyi's celebrated theorem asserting that an algebraic curve $C$
can be defined over $\bar{\Q}$ if and only if it admits a Belyi function,
i.e., a holomorphic  function $f \colon C \ra \PP^1$ whose  only critical values are in the set $\{0, 1, \infty\}$.
The main assertion concerns the functoriality of a certain Belyi function.

\begin{prop} \label{Belyi}
 Let $P \in \Q[x]$ be the minimal polynomial of $a \in \bar{\Q}$ and consider the field $L:=
\mathbb{Q}[x]/(P)$. Let $C_x$ be the hyperelliptic curve over $L$  
$$ w^2 = (z-x) (z + 2g) \Pi_{i=0}^{2g-1} (z-i) .$$ Then there is a rational function
$F_x \colon C_x
\rightarrow
\PP^1_L$ such that for each $a \in \C$ with $P(a) = 0$ it holds that  the rational function $F_a$ (obtained
under the specialization $ x \mapsto a$) is a Belyi function
 for $C_a$.
\end{prop}

 \begin{proof}
 Let $f_x \colon C_x \rightarrow \mathbb{P}^1_L$ be the hyperelliptic involution, branched in $\{-2g, 0,1,
\ldots , 2g-1, x \}$. Then $P \circ f_x$ has as critical values:

\begin{itemize}
\item the images of the critical values of $f_x$ under $P$, which are $\in \mathbb{Q}$,
\item the critical values $y$ of $P$, i.e. the zeroes of the discriminant $h_1(y)$ of $P(z) - y $ with respect 
to the variable $z$.
\end{itemize}

$h_1$ has degree $deg(P) - 1$, whence, inductively as in \cite{belyi},  we obtain $\tilde{f}_x := h \circ P \circ f$ whose
critical values are all contained in $\Q \cup \{ \infty \}$ (see
\cite{wolfart} for more details).
 If we take any root $a$ of $P$, then obviously
$\tilde{f}_a$ has the same critical values.

Let now $r_1, \ldots , r_n \in \Q$ be the (pairwise distinct) finite critical values of
$\tilde{f}_x$. We set:
$$ y_i := \frac{1}{\Pi_{j \neq i} (r_i - r_j)} \ .
$$

Let $N \in \N$ be a positive integer such that $m_i:=Ny_i \in \mathbb{Z}$. Then we have that the rational
function
$$ g(t):= \Pi_i (t-r_i)^{m_i} \in \mathbb{Q}(t)
$$ is ramified at most in $\infty$ and $r_1, \ldots r_n$. In fact, $g'(t)$ vanishes at most when $ g(t)= 0$ or at the points
where the
logarithmic derivative $ G(t) : = \frac{g'(t)}{g(t)} = \sum_i m_i (\frac{1}{t-r_i})$ has a zero. However, $G(t)$ has simple poles
at the $n$ points $r_1, \dots, r_n$  and by the
choice made we claim that  it has a zero of order $n$ at $\infty$. 

In fact, consider the polynomial $\frac{1}{N} G(t) \Pi_i (t - r_i)$, which has degree $\leq n -1 $ and equals
$$ \sum_i y_i  \Pi_{j \neq i}(t - r_j) = \sum_i  \Pi_{j \neq i}(t - r_j)\frac{1}{\Pi_{j \neq i} (r_i - r_j)} .$$
It takes value $1$ in each of the points $r_1, \ldots r_n$, hence it equals the constant $1$.

It follows that the critical values of $g \circ \tilde{f}_x$ are at most $0, ~ \infty, ~ g(\infty)$.

We set $F_x : = \Phi \circ g \circ \tilde{f_x}$ where $\Phi$ is the affine map $ z \mapsto 
 g(\infty)^{-1} z$, so that the critical values of $F_x$ are equal to $ \{ 0,1,\infty\}$.
 It is obvious by our construction that for any root
$a$ of $P$,
$F_a$ has the same critical values as $F_x$, in particular, $F_a$ is a Belyi function for $C_a$.

\end{proof}

Since in the sequel we shall consider the normal closure  (we prefer here, to avoid confusion, not to use the
term 'Galois closure' for the geometric setting) $ \psi_a \colon D_a \to \PP^1_{\C}$  of each of the functions
$F_a \colon C_a \to \PP^1_{\C}$, we recall in the next section the `scheme theoretic' construction of the normal
closure.

\section{Effective construction of normal closures}\label{nc}

In this section we consider algebraic varieties  over the complex numbers, endowed with their Hausdorff
topology, and, more generally,
 `good' covering spaces (i.e., between topological spaces which are locally arcwise  connected and semilocally
simply connected).

\begin{lem} Let $\pi \colon X \rightarrow Y$ be a finite `good'  unramified covering space of degree $d$ 
between connected spaces $X$ and $Y$. 

Then the normal closure $Z$ of $\pi \colon X \rightarrow Y$
(i.e., the minimal unramified covering of $Y$ factoring through $\pi$, and such that there exists
an action of a  finite group $G$  with $ Y = Z  / G$)  is
isomorphic to any connected component of 
$$W : = W_{\pi}:= (X \times _Y \ldots \times _Y X) \backslash \Delta \subset X^d
\backslash \Delta ,$$ where $\Delta := \{(x_1, \ldots, x_d) \in X^d  | 
 \exists  i\neq j ~, \   x_i =  x_j\}$ is the big diagonal.
\end{lem}

\begin{proof}
 Choose base points $x_0 \in X$, $y_0 \in Y$ such that $\pi (x_0) = y_0$ and denote by $F_0 $ the
fibre over $y_0$, $F_0 : = \pi^{-1}(\{y_0 \})$.

We consider the monodromy $\mu \colon \pi_1(Y, y_0)
\rightarrow
\mathfrak{S}_d = \mathfrak{S}( F_0)$ of the unramified covering $\pi$. The monodromy of $\phi \colon W
\ra Y$ is induced by the diagonal product monodromy $\mu^d \colon \pi_1(Y, y_0) \ra
\mathfrak{S}( F_0^d)$,  such that, for  $(x_1, \ldots , x_d) \in F_0^d$, we have $\mu^d(\gamma) (x_1,
\ldots , x_d) = (\mu(\gamma)(x_1), \ldots ,
\mu(\gamma)(x_d))$.

It follows that the monodromy of $\phi \colon W
\ra Y$,  $  \mu_W \colon \pi_1(Y, y_0) \ra
\mathfrak{S}( \mathfrak{S}_d )$ is given  by left translation
$\ \mu_W (\gamma) (\tau) = \mu (\gamma)\circ (\tau) $.

If we denote by $G := \mu ( \pi_1(Y, y_0) )\subset \mathfrak{S}_d $ the monodromy group, it follows right
away that  the components of $W$ correspond  to the cosets $G \tau$ of  $G$. Thus all the components
yield   isomorphic covering spaces.

\end{proof}

The theorem of Grauert and Remmert (\cite{g-r}) allows to extend the above construction to yield normal
closures of morphisms between normal algebraic varieties.

\begin{cor}\label{ncram} Let $\pi \colon X \rightarrow Y$ be a finite morphism between normal projective varieties, let
$B \subset Y$ be the branch locus of $\pi$ and set $X^0 := X \setminus \pi^{-1}(B)$,
$Y^0 := Y \setminus B$. 

If $X$  is connected, then the normal closure $Z$ of $\pi$ is isomorphic to any connected component of the
closure of $W^0 := (X^0 \times _{Y^0} \ldots \times _{Y^0} X^0) \backslash \Delta$ in the normalization
$W^n$  of $ W : = \overline{(X \times _Y \ldots \times _Y X) \setminus \Delta}$.

\end{cor}

\begin{proof}
The irreducible components of $W$ correspond to the connected components of $W^0$, as well as to the
connected components $Z$ of $W^n$.  So, our component $Z$ is the closure of a connected component
$Z^0$ of $W^0$. We know that the monodromy group $G$ acts on $Z^0$ as a group of covering
transformations and simply transitively on the fibre of $Z^0$ over $y_0$: by normality the action  extends
biholomorphically to $Z$, and clearly $ Z / G \cong Y$.

\end{proof}

\section{Faithful action of the absolute Galois group on the set of marked triangle curves  (associated to very special hyperelliptic curves)}

Let $a$ be an algebraic number, $g \geq 3$, and consider as in section \ref{1} the hyperelliptic curve
$C_a$ of genus $g$ defined by the equation
$$ w^2 = (z-a) (z + 2g) \Pi_{i=0}^{2g-1} (z-i) .$$  Let $F_a \colon C_a \rightarrow \PP^1$ be the Belyi function
constructed in proposition \ref{Belyi} and denote by $\psi_a \colon D_a \rightarrow \PP^1$ the normal closure
of $C_a$ as in corollary \ref{ncram}.

\begin{oss}  
\begin{enumerate} 
\item We denote by $G_a$ the monodromy  group of $D_a$ and observe that there is a subgroup
$H_a \subset G_a$ acting on $D_a$ such that $D_a /H_a \cong C_a$.
\item Observe moreover  that the degree $d$ of the  Belyi function $F_a$ depends not only on the degree of the
field etension $ [ \Q (a) : \Q]$, but much more on the height of the algebraic number $a$; 
 one may give an upper bound for the order of the group $G_a$ in terms of these.
 \end{enumerate}
\end{oss}

The pair $(D_a, G_a)$ that we get is a so called triangle curve, according to  the following definition
(see \cite{isogenous}):

\begin{df}\label{marked}
\begin{enumerate}
\item
A {\em marked variety} is a triple $(X, G, \eta)$ where $X$ is a projective variety and $\eta \colon  G \ra \Aut(X)$ is an injective
homomorphism
\item
equivalently, a marked variety is a triple $(X, G, \alpha)$ where $\alpha \colon X \times G \ra X$ is  an effective action of the group $G$ on $X$
\item  Two marked  varieties  $(X,G, \alpha)$, $(X',G', \alpha')$ are said to be {\em isomorphic} if there are isomorphisms $ f \colon X \ra X'$,
 and $\psi \colon G \ra G'$ transporting the action $\alpha \colon X \times G \ra X$ into the action $\alpha' \colon X' \times G' \ra X'$, 
 i.e., such that $$ f \circ \alpha = \alpha' \circ ( f \times \psi) \Leftrightarrow   \eta ' \circ \psi = \Ad(f) \circ \eta, \ \ 
 \Ad (f) (\phi) : = f \phi f^{-1}.$$
  \item
  Is $G $ defined as a subset of $\Aut(X)$, then the natural marked variety is the triple $(X,G,i)$, where $ i \colon G \ra \Aut(X)$
  is the inclusion map, and shall sometimes be denoted simply by the pair $(X,G)$.
\item  A marked curve  $(D,G, \eta)$ consisting of a smooth projective curve of genus $g$ and an effective action of the group $G$ on $D$
is said to be a {\em marked triangle curve of genus $g$} if $ D / G \cong \PP^1$ and the quotient morphism
 $p \colon D \ra  D / G \cong \PP^1$ is branched in three points.

 \end{enumerate}
\end{df} 

\begin{oss}
Observe that:

1)  we have a natural action of $\Aut(G)$ on marked varieties, namely $$ \psi (X,G,\eta) : = (X, G, \eta \circ \psi^{-1}).$$

2) the action of the group $\Inn(G)$ of inner automorphisms does not change the isomorphism class of $(X,G,\eta)$ since, for $\ga \in G$,  we may set 
$ f : =  (\eta(\ga))$, $\psi : = \Ad (\ga)$, and then  $\eta \circ \psi = \Ad(f) \circ \eta$, since 
$ \eta ( \psi (g)) = \eta ( \ga  g \ga^{-1} ) = \eta ( \ga) \eta( g ) (\eta(\ga)^{-1} ) = \Ad(f) (\eta(g)) $.

3) In the case where $ G = \Aut(X)$, we see that $\Out(G)$ acts simply transitively on the isomorphism classes of
the $\Aut(G)$-orbit of $  (X,G,\eta)$.
\end{oss}

 Consider now our triangle curve $D_a$: without loss of generality we may assume that the three branch points in $\PP^1$
 are $\{ 0,1, \infty\}$ and we may choose a monodromy representation
$$\mu \colon \pi_1(\PP^1 \setminus \{0, 1, \infty \}) \rightarrow G_a$$

  corresponding to the normal ramified
covering $\psi_a \colon D_a \rightarrow \PP^1$.   Denote further by $ \tau_0, ~ \tau_1, ~ \tau_{\infty}$
 the images of geometric loops around $0, ~ 1, ~
\infty$ . Then we have that
$G_a$ is  generated by $ \tau_0, ~ \tau_1, ~ \tau_{\infty}$ and
$\tau_0 \cdot \tau_1 \cdot \tau_{\infty}= 1$. By Riemann's existence theorem the datum of these
three generators of the group $G_a$ determines a marked triangle curve (see \cite{isogenous}, 
 \cite{bcg}).

We can phrase our previous considerations in a theorem, after  preliminarily observing:

\begin{oss}
1) $\s \in \Aut (\C)$ acts on $ \C [z_0, \dots z_n]$,
by sending $P (z) = \sum_{I=(i_0, \dots, i_n)} \ a_I z ^I
\mapsto  \s  (P) (z) : = \sum_{I=(i_0, \dots, i_n)} \ \s  (a_I) z ^I$.

2) Let $X$ be  a projective variety
$$X  \subset  \PP^n_\C,  X : = \{ z | f_i(z) = 0 \  \forall i \}.$$
The action of $\s $ extends coordinatewise to $ \PP^n_\C$,
and carries $X$ to the set $\sigma (X)$ which is another variety, denoted $X^{\s }$,
and called the {\em conjugate variety}. In fact, since $f_i(z) = 0 $ implies
$\s  (f_i)(\s  (z) )= 0 $, one has that
$$  X^{\s }  = \{ w | \s  (f_i)(w) = 0 \ \forall i \}.$$
3) Likewise, if $f  \colon X\ra Y$ is a morphism, its graph $\Ga_f$ is a subscheme of $ X \times Y$, hence we get a conjugate
morphism $f^{\s } \colon X^{\s } \ra Y^{\s }$. 

4) Similarly, if $ G \subset \Aut(X)$, and $ i  \colon G \ra \Aut(X)$ is the inclusion, then $\s$ determines
another marked variety 
$(X^{\s },  G , \Ad (\s) \circ i )$, image of $(X, G, i)$. 

In other  words, we have $ G^{\s }  \subset \Aut (X^{\s } )$ in such a way that, if we identify
$G$ with $ G^{\s }$ via $ \Ad (\s) $, then $ (X / G)^{\s } \cong X^{\s }/ G.$

\end{oss}

\begin{teo}\label{markedtriangle} To any algebraic number $a \notin \Z$ there corresponds, through a canonical procedure (depending on an integer $g \geq 3$), a marked triangle curve $(D_a, G_a)$.

This correspondence yields a faithful action of the  absolute Galois group $\Gal(\bar{\Q} /\Q)$ on the  set of isomorphism classes
of marked triangle curves.
\end{teo}

\begin{proof}
Let $(D,G)$ be a marked triangle curve,  and $\sigma \in \Gal(\bar{\Q} /\Q)$: extend  $\sigma$ 
to  $\sigma \in \Gal(\C /\Q)$ and take the transformed curve $D^{\s}$ and the transformed
graph of the action, a subset of  $D^{\s}\times D^{\s} \times G$.

Since there is only a finite number of isomorphism classes of such pairs $(D,G)$ of  a fixed genus $g$
and with fixed group $G$, it follows that $D$ is defined over $\bar{\Q}$ and the chosen extension
of $\s$ does not really matter.

Finally, apply the action of $\s$ to the triangle curve $(D_a, G_a)$ and assume that 
the   isomorphism class of $(D_a, G_a)$ is fixed
 by the action. This means then, setting $ b : = \s (a)$,  that there is an isomorphism
$$f \colon D_a \ra D_b = D_{\s (a)} = D_a^{\s}$$ 
such that $ \Ad (f) = \Ad(\s )$.

In other words, $\s$ identifies $G_a$ with $G_b$ by our assumption and the two actions of
$G_a$ on $D_a$ and $D_b$ are transported to each other by $f$.

It suffices  to show that under the isomorphism $f$ the subgroup $H_a$ corresponds to the subgroup
$ H_{b}$ (i.e., $\Ad (f) (H_a) = H_b$). 

Because then we conclude that, since $C_a = D_a / H_a$, $C_b = D_b / H_b$, $f$ induces an isomorphism 
of  $C_a$ with  $C_{b}$.

And then  by proposition \ref{Belyi} we conclude that $a = b$.

We use now that $ \Ad (f) = \Ad(\s )$, so it suffices to show the following

\begin{lem}\label{H}
 $\Ad (\s ) (H_a) = H_b$.
\end{lem}
{\em Proof of the Lemma.}

Let $K$ be the Galois closure of the field $L$
(= splitting field of the field extension $\Q \subset L$), and view $L$ as embedded in $\C$ under the
isomorphism sending $x \mapsto a$.

Consider the curve $\hat{C}_x$  obtained from $C_x$ by scalar extension $\hat{C}_x : = C_x \otimes_L K$.
Let also $\hat{F}_x : = F_x \otimes_L K$ the corresponding Belyi function with values in $\PP^1_K$.

Apply now the effective construction of the normal closure of  section \ref{nc}: hence, taking a connected
component of  
$ (\hat{C}_x \times _{\PP^1_K} \ldots \times _{\PP^1_K} \hat{C}_x) \setminus \Delta $  we obtain  a
curve $D_x$ defined over $K$.

Note that $D_x$ is not geometrically irreducible, but, once we tensor with $\C$, it splits into several
components which are Galois conjugate and which are isomorphic to the conjugates of $D_a$.

Apply now the Galois automorphism $\sigma$ to the triple $D_a \to C_a \to \PP^1$. Since the triple is
induced by the triple  $D_x \to C_x \to \PP^1_K$ by taking a tensor product $ \otimes_K \C$ via the
embedding sending $ x \mapsto a$, and the morphisms are induced by the composition of the inclusion $D_x
\subset (C_x)^d$ with the coordinate projections, respectively by the fibre product equation, it follows
from proposition \ref{Belyi} that $\sigma$ carries the triple $D_a \to C_a \to \PP^1$ to the triple $D_b \to
C_b \to \PP^1$.

\hfill{\em QED.}

\end{proof}
If we want to interpret our argument in terms of Grothendieck's \'etale fundamental group, we  define
$C^0_x : = F_x^{-1} (\PP^1 \setminus \{0,1,\infty \})$, and accordingly  $\hat{C}^0_x$ and $D^0_x$.

There are the following  exact sequences for the  Grothendieck \'etale fundamental group (compare
Theorem 6.1 of \cite{sga1}):
$$ 1 \ra \pi_1^{alg} (D^0_a) \ra  \pi_1^{alg} (D^0_x) \ra \Gal (\bar{\Q}/ K) \ra 1$$
$$ 1 \ra \pi_1^{alg} (C^0_a) \ra  \pi_1^{alg} (C^0_x) \ra \Gal (\bar{\Q}/ K)  \ra 1$$
$$ 1 \ra \pi_1^{alg} (\PP^1_{\C}  \setminus \{0,1,\infty \}) \ra  
\pi_1^{alg} (\PP^1_{K}  \setminus \{0,1,\infty \}) \ra \Gal (\bar{\Q}/ K) \ra 1$$

where $H_a$ and $G_a$ are the respective factor groups for the (vertical) inclusions of the left hand sides,
corresponding to the first and second sequence, respectively  to the first and third sequence.

On the other hand, we also have the exact sequence
$$ 1 \ra \pi_1^{alg} (\PP^1_{\C}  \setminus \{0,1,\infty \}) \ra  
\pi_1^{alg} (\PP^1_{\Q}  \setminus \{0,1,\infty \}) \ra \Gal (\bar{\Q}/ \Q) \ra 1.$$

The finite quotient $G_a$ of $\pi_1^{alg} (\PP^1_{\C}  \setminus \{0,1,\infty \})$ (defined over $K$) is
sent by $\sigma \in \Gal (\bar{\Q}/ \Q)$ to another quotient, corresponding to $D_{\sigma (a)}$, and the
subgroup $H_a$, yielding the quotient
$C_a$, is sent to the subgroup  $H_{\sigma (a)}$.

\begin{oss}
Assume  that the  two triangle curves $D_a $ and $D_b = D_{\s (a)}$ are isomorphic
 through a complex  isomorphism  $ f  \colon D_a \ra D_b$ (but  without
that   necessarily $ ( f, \Ad (\s ))$ yields an isomorphism of
marked triangle curves $ (D_a, G_a) $, $ (D_b, G_b) $).
 
We define $ \psi : G_a \ra G_a$ to be equal to
$ \psi : =  \Ad ( \s ^{-1} \circ f)$.

Then $\Ad (f) = \Ad (\s ) \circ \psi $ and applying to $y \in D_b, y = f(x)$ we get
$$ \Ad (f)  (g) (y) =  (\Ad (\s ) \circ \psi)(g) (y) \Leftrightarrow  f (g (x)) =  (\Ad (\s ) \circ \psi)(g) (f(x)).$$
 
Identifying  $G_a$ with $G_b$ under  $ \Ad (\s)$, one can interpret the above formula as asserting
that $f $ is only `twisted' equivariant ($ f ( g (x)) "="  \psi (g) (f (x))$.
\end{oss}

\begin{prop}\label{innerOK}
Assume  that the  two triangle curves $D_a$ and $D_b = D_{\s (a)}$ are isomorphic under a
complex  isomorphism $f$
and that the above automorphism 
$\psi \in \Aut (G)$ such that $ f ( g (x)) =  (\Ad (\s ) \circ \psi)(g) (f(x))$ is inner. 

Then $C_a \cong C_b$, hence $ a = b$.

\end{prop}
\begin{proof}
If $\psi$ is inner, then the marked triangle curves $(D_a, G_a)=  (D_a, G_a, i_a)$ ($ i_a$ being the inclusion map 
of $G_a \subset \Aut (D_a)$), and its transform by $\s$,
$(D_b, G_a^{\s}) = (D_b, G_a, \Ad (\s) \circ i_a)$ are isomorphic.

Then the argument of   theorem \ref{markedtriangle} implies that $C_a \cong C_b$, hence $ a = b$.

\end{proof}

We pose here the following conjecture, which is a strengthening of the previous theorem \ref{markedtriangle}

\begin{conj}\label{unmarked}(Conjecture 2.13 in \cite{catbumi})
The absolute Galois group $\Gal(\bar{\Q} /\Q)$
acts faithfully
on the  set of isomorphism classes of (unmarked) triangle curves.
\end{conj} 

The following definition will be useful in the proof of theorem \ref{isogenous}.
\begin{df}\label{fullmarked} Let $a \in \bar{\Q}$ be an algebraic number and let $(D_a, G_a)$ be the associated marked triangle curve obtained by the canonical procedure above (depending on an integer $g \geq 3$). Then $(D_a, \Aut(D_a))$ is called the {\em fully marked} triangle curve associated to $a$.
\end{df}
\begin{oss}
If we consider instead of $(D_a, G_a)$ the fully marked triangle curve $(D_a, \Aut(D_a))$ we have also the subgroup $H_a \leq \Aut(D_a)$ such that $D_a / H_a = C_a$, where $C_a$ is the very special hyperelliptic curve associated to the algebraic number $a$. 
\end{oss}

The same proof as the proof of theorem \ref{markedtriangle} gives

\begin{teo}\label{fullmarkedtr}
To any algebraic number $a \notin \Z$ there corresponds, through a canonical procedure (depending on an integer $g \geq 3$), a  fully marked triangle curve $(D_a, \Aut(D_a))$.

This correspondence yields a faithful action of the  absolute Galois group $\Gal(\bar{\Q} /\Q)$ on the  set of isomorphism classes
of fully marked triangle curves.

\end{teo}

\section{Connected components of moduli spaces associated to very special hyperelliptic curves}

Fix now an integer $ g \geq 3$, and  another integer $g' \geq 2$.

Consider now   all the algebraic numbers $ a \notin \Q$ and all the possible smooth complex curves
$C'$ of genus $g' $, observing  that the fundamental group of $C'$ is  isomorphic to the standard group 
$$\pi_{g'} : = \langle \alpha_1,
\beta_1, \ldots ,
\alpha_{g'}, \beta_{g'} | \Pi_{i=1}^{g'} [ \alpha_{i}, \beta_{i}] = 1 \rangle.$$

Since $g' \geq 2$ and $G_a$ is $2-$generated  there are plenty of 
epimorphisms  (surjective homomorphisms) $ \mu \colon \pi_{g'} \ra  G_a$.
For instance it suffices to consider the epimorphism 
$\theta \colon \pi_{g'} \ra \F_{g'}$ from $\pi_{g'}$ to the free group $\F_{g'}:=<\lambda_1, \ldots ,
\lambda_{g'}>$ in $g'$ letters given by $\theta(\alpha_i) = \theta(\beta_i) = \lambda_i$, $\forall$ $1
\leq i \leq g'$, and to 
 compose $\theta$ with the surjection $\phi \colon \F_{g'} \rightarrow G_a$, given by $\phi (\lambda_1) =
\tau_0$, $\phi(\lambda_2) = \tau_1$, and $\phi (\lambda_i)= 1$ for $3\leq i \leq g'$.

Consider all the possible  epimorphisms   $ \mu \colon \pi_{g'} \ra  G_a$. Each such $\mu$ gives a
normal unramified covering $D' \rightarrow C'$ with  monodromy   group $G_a$.

Let us recall now the basic definitions underlying our next construction:  the theory of
surfaces  isogenous  to a product, introduced in \cite{isogenous}(see also
\cite{cat03}), and which  holds more generally  for 
varieties isogenous to a product.

\begin{df} 
\begin{enumerate}
\item A surface {\em isogenous to a (higher) product}
is a compact complex projective surface $S$ which
 is  a quotient
$S = (C_1 \times C_2) /G$ of a product of curves of resp. genera 
$g_1, g_2 \geq 2$ by the
free action of a finite group $G$. It is said to be {\em unmixed} if the embedding
$ i \colon G \ra \Aut(C_1 \times C_2)$ takes values in the subgroup (of index at most two) $ \Aut(C_1) \times \Aut(C_2)$.
\item  A  {\em Beauville  surface}   is a surface  isogenous to a (higher) product
which is {\em rigid}, i.e., it has no nontrivial deformation. This 
amounts to the condition,
in the unmixed case,  that  $ (C_i, G)$ is a triangle curve.
\item An {\em \'etale marked surface} is a triple  $(S',  G, \eta)$ such that the action of $G$ is fixpoint free.
An  \'etale marked surface
can also be defined as a quintuple $(S, S', G, \eta, F)$ 
where $ \eta \colon G \ra \Aut(S')$ is an effective free action, and $F  \colon S \ra S'/G$
is an  isomorphism.
\end{enumerate}
\end{df}

Observe that a surjection of the fundamental group $r \colon \pi_1 (S, y) \ra G$
determines an \'etale marked surface. Once a base point $y$ is fixed, then the marking
provides the desired surjection $r$. Moving  the base point $y$ around amounts to
replacing $r$ with the composition $ r \circ \Ad (\ga)$, for $\ga \in \pi_1 (S, y)$. However,
$$  (r \circ \Ad (\ga) )(\delta) = r (\ga \delta \ga^{-1}) = (\Ad ( r (\ga)) \circ r )(\delta). $$

Therefore we see that we have to divide by the group $\Inn (G)$ acting on the left.

In this case the associated subgroup of the covering is a normal subgroup, hence uniquely determined,
independently of the choice of a base point above $y$; however, the corresponding isomorphism of the
quotient group with $G$ changes, and as a result 
 the epimorphism $r$ is modified by an inner automorphism of $G$.
 On the other hand the action of nontrivial elements in $\Out (G) : = \Aut(G) / \Inn (G)$ may transform
  the marking into a non isomorphic one.

\begin{oss}
Consider the  coarse moduli space $\frak M_{x,y}$ of canonical models of surfaces of general 
type $X$ with $\chi (\hol_X) = x, K^2_X = y$.  Gieseker (\cite{gieseker}) proved that $\frak M_{x,y}$ is a quasi-projective variety.

We denote by $\frak M$ the disjoint union $\cup_{x,y \geq 1} \frak M_{x,y}$, and we call it the {\em moduli space of
surfaces of general type}.

Fix a finite group $G$ and consider the moduli space $\hat{ \frak M}^G_{x,y}$ for  \'etale marked surfaces 
$(X, X', G, \eta, F)$, where the isomorphism class $[X] \in \frak M_{x,y}$.

This moduli space $\hat{ \frak M}^G_{x,y}$ is empty if there is no  surjection  $r \colon \pi_1 (X, y) \ra G$,
otherwise we obtain that $\hat{ \frak M}^G_{x,y}$ is a finite \'etale covering space of $\frak M_{x,y}$
with fibre over $X$ equal to the quotient set  
$$ \Epi ( \pi_1 (X, y), G) / \Inn (G) .$$ 

By the theorem of Grauert and Remmert ( \cite{g-r}) $\hat{ \frak M}^G_{x,y}$ is a quasi-projective variety.

\end{oss}
Recall the following result concerning  surfaces isogenous to a product  (\cite{isogenous},
\cite{cat03}):

\begin{teo}
Let $S = (C_1 \times C_2) /G$ be a surface isogenous to a product. 
Then any surface $X$ with the
same topological Euler number and the same fundamental group as $S$
is diffeomorphic to $S$. The corresponding subset of the moduli space
$\frak M^{top}_S = \frak M^{diff}_S$, corresponding to surfaces homeomorphic,
resp, diffeomorphic to $S$, is either irreducible and connected
or it contains
two connected components which are exchanged by complex
conjugation.
\end{teo}

If $S$ is a Beauville surface (i.e., $S$ is rigid) this implies: $X \cong S$ or $X \cong \bar{S}$.
It follows also that a Beauville surface is defined over $\bar{\Q}$,
whence the Galois group $\Gal (\bar{\Q} / \Q)$ acts on the
discrete subset of the moduli space $\mathfrak M$ of surfaces of general type  corresponding
to Beauville surfaces.

It is tempting to make the following
\begin{conj}(Conjecture 2.11 in \cite{catbumi})
The absolute Galois group $\Gal (\bar{\Q}  /\Q)$ acts
faithfully on the
discrete subset of the moduli space $\mathfrak M$ of surfaces of general type corresponding
to Beauville surfaces.

\end{conj}

\begin{df} Let $\frak N_a$ be the subset of the moduli space of surfaces of general type given by surfaces
isogenous to a product of unmixed type
$S \cong (D_a \times D' )/G_a$, where $D_a, D'$ are as above (and the group $G_a$ acts by the diagonal
action).
\end{df}

From \cite{isogenous} and especially Theorem 3.3 of \cite{cat03} it follows:

\begin{prop} For each $ a \in \bar{\Q}$,  $\frak N_a$ is a union of connected components of the moduli
spaces of surfaces of general type.

Moreover, for $\sigma \in \Gal ( \bar{\Q} / \Q)$, $\sigma (\frak N_a )= \frak N_{\sigma(a)}$.
\end{prop}

\begin{proof}
Since $D_a$ is a triangle curve, the pair $ (D_a, G_a)$ is rigid, whereas, varying $C'$ and
$\mu$, we obtain the full union of the moduli spaces for the pairs $(D', G_a)$, corresponding to the
possible free topological actions of the group $G_a$ on a curve $D'$ of genus $ |G_a| (g'-1) + 1$.

Thus, the surfaces $S \cong (D_a \times D' )/G_a$  give, according to the cited theorem 
  3.3 of \cite{cat03},  a union of connected components of the moduli space $\frak M$  of surfaces of general type.

Choose now  a surface  $S$ as above (thus, $[S] \in \frak N_a$) and apply the field automorphism $\s \in \Aut (\C)$
 to a point of the Hilbert scheme corresponding to the 5-canonical image
of $S$ (which is isomorphic to $S$, since the canonical divisor of $S$ is ample). We obtain a surface which
we denote by $  S ^{\sigma}$.

By taking the fibre product of $\sigma$ with  $D_a \times D' \ra S$ it follows that 
$ S ^{\sigma}$ has an \'etale covering with group $G_a$ which is the product
$ (D_a )^{\sigma} \times ( D') ^{\sigma} $.

Recall  that  $(C_a)^{\sigma} = C_{\sigma(a)} $ (since $\sigma (a)$ corresponds to another embedding of the
field $L$ into $\C$), and recall  the established equality for Belyi maps $(F_a) ^{\sigma} = F_{\sigma(a)} $, which implies $(D_a )^{\sigma} =
D_{\sigma(a)} $.

On the other hand, the quotient of $( D') ^{\sigma}$ by the action of the group $G_a$  has genus equal to
the dimension of the space of invariants 
$ \dim (H^0 (\Omega^1_{  ( D') ^{\sigma}})^{G_a}  )$, but this dimension is the same as $g' =  \dim (H^0
(\Omega^1_{  D'})^{G_a}  )$. Hence the action of $G_a$ on $( D') ^{\sigma} $ is also free (by Hurwitz'
formula), and we have shown that $ S^{\sigma}$ is a surface whose moduli point is in
$ \frak N_{\sigma(a)}$.

Finally, since the subscheme  of the Hilbert scheme corresponding to these points is defined over 
$\Q$, it follows that the action on the set of connected components of this subscheme,
which is the set of connected components of the moduli space, depends only on
the image $\sigma \in \Gal ( \bar{\Q} / \Q)$.

\end{proof}

Let us explain the rough idea for our strategy:   for each $\s \in \Gal ( \bar{\Q} / \Q)$ which is nontrivial
we would like to find $ a$ such that, setting $ b : = \s (a) $:

(**) \ \ $ a \neq b$ and $\frak N_a$ and $\frak N_b$ do not intersect.

If (**) holds, then we can easily conclude that
$\s$ acts nontrivially on the set  $\pi_0 ( \frak M)$ of connected components
of  $\frak M$. If (**) does not hold for each $a$, the strategy must  be changed and becomes a little bit more complicated.

Observe that the condition that $\frak N_a$ and $\frak N_b$ intersect  implies,
by the structure theorem for surfaces isogenous to a product, only the weaker statement that
the two triangle curves $D_a$ and $D_b$ are isomorphic.  

In order to resort to the result established in the previous section we first of all 
consider  some connected components of moduli spaces of  \'etale marked surfaces, specifically of   \'etale marked surfaces isogenous to a product.

Let therefore $ S = (C_1 \times C_2)/G$ be a surface isogenous to a product, of unmixed type.
Then the conjugate surface $ S ^{\s}$ has an \'etale cover with group $G$,
and we see that $\s$ acts on the  \'etale marked surface $(S,C_1 \times C_2, G, \eta,F)$ 
carrying it to  $(S^{\s}, C_1^{\s} \times C_2^{\s}, G^{\s}, \eta^{\s},F^{\s})$.
In particular, $ \Ad (\s )$ identifies the group $G$ acting on $C_1 \times C_2$ with the one
acting on $  C_1^{\s} \times C_2^{\s}$.

Now, $S^{\s}$ belongs to the same connected component of $S$, or to its complex
conjugate, if and only if there exists an isomorphism $ \Phi \colon \pi_1(S^{\s}) \ra \pi_1 (S)$.

Identifying $G$ and $G^{\s}$ via $ \Ad (\s )$,  the surjection $r \colon \pi_1 (S) \ra G$,
whose kernel is  $\pi_1 (C_1 \times C_2)$, 
yields a second surjection $ r \circ \Phi \colon \pi_1(S^{\s}) \ra G$, which has  kernel  $\pi_1 (C_1 \times C_2)^{\s}$,
by the unicity of the minimal realization of a surface isogenous to a product (\cite{isogenous}, prop. 3.15).

Hence  $ r \circ \Phi $ differs from the surjection
$r^{\s}$ via an automorphism $ \psi \in \Aut (G)$ such that

$$ r \circ \Phi = \psi \circ  r^{\s}.$$

The condition that moreover we get the same  \'etale marked  surface is that
the automorphism $\psi $ is inner. The above argument yields now the following

\begin{teo} \label{marked isogenous}The absolute Galois group $\Gal(\bar{\Q} /\Q)$ acts faithfully on the  set of connected
components of the (coarse) moduli space of  \'etale marked  surfaces isogenous to a higher product. 
\end{teo}

\begin{proof}
Given $ a \in \bar{\Q}$, consider a connected component $\hat{ \frak N}^{\rho}_a$ of the space of
$G$-marked surfaces of general type $\frak M^G_{x,y}$
corresponding to a given homomorphism 
$$  \mu \colon \pi_{g'} \ra G: = G_a.$$

This homomorphism has a kernel isomorphic to $\pi_{g_2}$, and conjugation by elements of  $\pi_{g'} $
determines a homomorphism 
$$ \rho \colon G \ra \Out ^+ ( \pi_{g_2}) = \Map_{g_2} $$ 
which is well defined, up to conjugation in the mapping class group $\Map_{g_2}$
($\rho$ is the topological type of the action of $G$).

Our   theorem
 follows now from the following

\noindent
{\bf Main Claim:} if $\hat{ \frak N}^{\rho}_a  = \sigma (\hat{ \frak N}^{\rho}_a)$, then necessarily $a  = \sigma (a)$.

Our assumption says that there are  two curves $C, C' $ of genus $g'$, and two respective covering curves
$C_2$, $C_2'$, with group $G_a$ and monodromy  type $\mu$ (equivalently, with topological type $\rho$ of the action
of $G_a$), such that there exists an isomorphism

$$ f : D_a^{\s} \times  C_2^{\s} \ra D_a \times C_2'$$

commuting with the action of $G_a$ on both surfaces.

By the rigidity lemma 3.8 of  \cite{isogenous}, $f$ is of product type, and since one action is not free while the other is free,
we obtain that $ f = f_1 \times f_2$, where $f_1 : D_a^{\s}  \ra D_a$ commutes with the $G_a$ action.

Therefore the marked triangle curves $(D_a, G_a, i_a)$ and  $(D_a^{\s} , G_a, \Ad (\s)  i_a)$ are isomorphic and
by Theorem  \ref{markedtriangle} we get $ a = \s (a)$.

\end{proof}

\subsection{What happens, if we forget the marking?} \ 

\noindent
Assume now  that $\s$ acts as the identity on the subset  $\pi_0( \frak N_a)$ of   $\pi_0( \frak M)$,
whose  points correspond to the connected components  $\frak N^{\rho}_a$ (image of $\hat{\frak N}^{\rho}_a$ in the moduli space
$\frak M$ of surfaces of general type).

 Then we use the following trick:
let $\la \in \Aut (G)$, and consider now the epimorphism $ \la \circ \mu$, to which corresponds  
the connected component $\frak N^{\rho  \la^{-1}}_a$ of the moduli space.

The component contains a surface $S_{\la}$ which is the quotient of the same product of curves
$C_1 \times C_2$, but where the action of $G$ is different, since we divide  by another subgroup,  the subgroup

$$ G (\la) : = \{ ( g, \la(g)) \subset G \times G\}. $$

Therefore the Galois action on $C_1 \times C_2$  is always the same, and we get, by the 
 above assumption,  that there is an automorphism $\psi_{\la}$ of $G$, induced by an isomorphism of
 $ \pi_1 (S_{\la} ) $ with $ \pi_1 (S_{\la}^{\s} ) $. 
 
 Indeed, this isomorphism of fundamental groups is induced by an isomorphism of 
 $ S_{\la} $ with the conjugate $(S''_{\la})^{\s} $ of another surface $S''_{\la} $ in the connected component;
 this isomorphism lifts to an isomorphism   of product type 
 $$ f_1 \times f_2 : C_1 \times C_2 \cong C_1^{\s} \times (C''_2)^{\s}.$$
 
 Notice that, for each $\la$, $(S''_{\la})^{\s} $ is a quotient of $C_1^{\s} \times (C''_2)^{\s}$.

 Identifying these two surfaces to $ S_{\la}  = (C_1 \times C_2 )/G$ via this isomorphism, we get that
the Galois automorphism $\s$ acts on $G \times G$ by a product automorphism $\psi_1 \times \psi_2$,
 and the automorphism  $\psi_{\la}$ of $G$ is induced by the identification of $G \cong G (\la)$
 given by the first projection.
 
 Note  that, while $\psi_1$ is unique, $\psi_2$ is only defined up to an inner automorphism,
 corresponding to an automorphism of $C_2$ contained in $G$.
 
 We must now have that 
 $$( \psi_1 \times \psi_2) (G (\la) ) = G (\la)  \Leftrightarrow (\psi_1(g), \psi_2(\la (g))) \in G (\la)    \Leftrightarrow $$
$$   \Leftrightarrow  \psi_2(\la (g))) = \la (  \psi_1(g))   \forall g \in G. $$

By setting $\la = Id$, we obtain $\psi_1 = \psi_2$, and using
$$ \psi_2 \circ \la  = \la \circ  \psi_1 $$

 we reach the following conclusion 

\begin{prop}\label{center} $\psi_1$  lies in the centre $ Z ( \Aut (G))$ of   $ \Aut (G)$, 
in particular the class  $[\psi_1] \in \Out (G)$ lies in the centre $ Z ( \Out (G))$.
\end{prop}

Clearly, if this class is trivial, then the  triangle curves $(D_a, G)$ and $(D_b,G)$  differ by an inner automorphism
of $G$ and we conclude by proposition \ref{innerOK}  that $C_a \cong C_b $, hence $a=b$, a contradiction.

Hence we may assume that the class $[\psi_1] \in  Z ( \Out (G))$ is nontrivial. 

We are now ready for the proof of 

\begin{teo} \label{isogenous2}The absolute Galois group $\Gal(\bar{\Q} /\Q)$ acts faithfully on the  set of connected
components of the (coarse) moduli space of    surfaces of general type. 
\end{teo}

There are two main intermediate results, which obviously together  imply theorem \ref{isogenous}.

For the first we need a new definition.

\begin{df}
Let $ a\in \bar{\Q} \setminus \Q$ and define $\tilde{G}_a : = \Aut (D_a)$. 

Given a 
surjective homomorphism $$  \tilde{\mu} \colon \pi_{g'} \ra \tilde{G}_a ,$$
with topological type $\tilde{\rho}$, consider all the \'etale covering spaces $C_2 \ra C_2 / \tilde{G}_a = C'$
of curves $C'$ of genus $g'$ with this given topological type.

 Consider then the connected component $\tilde{ \frak N}^{\tilde{\rho}}_a$ of the moduli space of
surfaces of general type $\frak M$
corresponding to surfaces isogenous to a product of the type
$$ S = (D_a \times C_2) /   \tilde{G}_a.$$

\end{df}

\begin{prop}\label{Abeliankernel}
Let $\mathfrak K$ be the kernel of the action of $\Gal(\bar{\Q} /\Q)$ on $\pi_0 (\mathfrak M)$. Then
$\mathfrak K$ is an abelian subgroup.
\end{prop}

\begin{proof}
We want to embed the kernel $\mathfrak K$ in  an abelian group, e.g. a  direct product of  groups of the form $ Z ( \Out (G))$, using proposition \ref{center}.

Assume that $\s $ lies in the kernel $\mathfrak K$. Then, for each algebraic number $a$, and every $\tilde{\rho}$ 
as above, $\s$ stabilizes the connected component  $\tilde{ \frak N}^{\tilde{\rho}}_a$ .

Let us denote here for simplicity $G : =  \tilde{G}_a.$

Hence to $\s$ we associate an element $[\psi_1] \in  Z ( \Out (G))$, which is nontrivial if and only if $\s (a) \neq a$.

Therefore it  suffices to show that, for a fixed $a \in \bar{\Q}$, and $\tilde{\rho} \colon G \ra \Map_{{g}_2} = \Out (\pi_{{g}_2})$,
the function explained before stating proposition \ref{center}
$$ \s \mapsto  [\psi_1] \in   Z ( \Out (G))$$
is a homomorphism. 

Observe that in fact, there is a dependence of $\psi_1$ on $\sigma$ and on the algebraic number $a$. To stress these dependences, we change the notation and denote the isomorphism $\psi_1$
corresponding to $\s$ and $a$,  by $\psi_{\s, a}$, i.e., 
$$\psi_{\s, a} (g) = \Phi_1^{-1} \circ g^{\s} \circ \Phi_1 = \Phi_1^{-1} \circ \s g  \s^{-1} \circ \Phi_1 ,$$
where $\Phi_1 \colon D_a \ra D_{\s(a)}$ is the isomorphism induced by the fact that $\sigma$ stabilizes the component 
$\tilde{ \frak N}^{\tilde{\rho}}_a$.

Since $D_a$ is fully marked, whenever we take  another isomorphism $\Phi \colon D_a \rightarrow D_{\s(a)}$  we have that $  (\Phi)^{-1} \circ\Phi_1\in \Aut(D_a) = G$. Therefore,  since we work in $\Out(G_a)$,  $\psi_{\s, a}$ does not depend on the chosen isomorphism $\Phi \colon D_a \rightarrow D_{\s(a)}$.

Let now $\sigma, \tau$ be elements of $\mathfrak{K}$. We have  then (working always up to inner automorphisms of $G$):
\begin{itemize}
\item $\psi_{\s, a} = \Ad(\varphi^{-1}\sigma)$, for any isomorphism $\varphi \colon D_a \ra D_{\s (a)}$, and any algebraic number $a$;
\item $\psi_{\tau, a} = \Ad(\Phi^{-1}\tau)$, for any isomorphism $\Phi \colon D_a \ra D_{\tau (a)}$, and any algebraic number $a$; 
\item $\psi_{  \tau \s, a} = \Ad(\Psi^{-1}\tau\s)$, for any isomorphism $\Psi \colon D_a \ra D_{\tau \s (a)}$.
\end{itemize}

We can choose $\Psi := \varphi^{\tau} \circ \Phi$, and  then we see immediately that 
\begin{multline}
\psi_{\tau \s, a} = \Ad((\varphi^{\tau} \circ \Phi)^{-1} \tau \s) = \Ad(\Phi^{-1}(\varphi^{\tau})^{-1} \tau \sigma) = \\
= \Ad(\Phi^{-1} \tau \varphi^{-1} \tau^{-1} \tau \s) = \Ad(\Phi^{-1}\tau) \Ad(\varphi^{-1}\sigma) = \psi_{\s, a} \circ \psi_{\tau , a} .
\end{multline}

This shows that the injective map 
$$
\mathfrak{K} \ra \prod_{a \in \bar{\QQ}} \prod _{\tilde{\rho}} Z(\Out( \tilde{G}_a)),
$$
is in fact a group homomorphism. Therefore $\mathfrak{K}$ is abelian (as subgroup of an abelian group).

\end{proof}

\begin{prop}\label{noAbeliankernel}
Any abelian normal subgroup $\mathfrak K$  of $\Gal(\bar{\Q} /\Q)$ is trivial.
\end{prop}

\begin{proof}
Let $ N \subset \bar{\Q} $ be the fixed subfield for the subgroup $\mathfrak K$ of $\Gal(\bar{\Q} /\Q)$.

$N$ is 
a Galois extension of the  Hilbertian field $\Q$ and $N$, if  $\mathfrak K$ is not trivial,  is not separably
closed. Hence,  by proposition 16.11.6 of \cite{F-J} then $\Gal(N)$ is not prosolvable, in particular, 
$\mathfrak K = \Gal(N)$ is not abelian, a contradiction.

\end{proof}

Theorem \ref{isogenous} has the following consequence:

\begin{teo} \label{fundamentalgroup}  If $\sigma \in \Gal(\bar{\Q} /\Q)$ is not in the conjugacy class of complex conjugation $\gc$,
then there exists a surface isogenous to a product $X$ such that $X$ and the Galois conjugate surface $X^{\sigma}$  have non isomorphic fundamental groups. 
\end{teo}
\begin{proof} By a  theorem of Artin (see cor. 9.3 in \cite{Lang})) we know that any $\s$ which is not in the conjugacy class of $\mathfrak c$ has infinite order.

By theorem \ref{isogenous} the orbits of $\s$ on the subset of $\pi_0 (\mathfrak M)$ corresponding to
the union of the $\mathfrak N_a$'s have unbounded cardinality, otherwise there is a power 
of $\s$ acting trivially,  contradicting the statement of \ref{isogenous}.

Take now an orbit with three elements at least: then we get surfaces $S_0$, $S_1 : = S_0^{\s}$, $S_2 : = S_1^{\s}$,
which belong to three different components. Since we have at most two different connected
components where the fundamental group is the same, we conclude that either  $\pi_1 (S_0) \neq  \pi_1 (S_1) $
or  $\pi_1 (S_1) \neq  \pi_1 (S_2) $.

\end{proof}

 Observe that $X_a$ and $(X_a)^{\sigma}$ have isomorphic
 Grothendieck \'etale fundamental groups. In particular, the profinite completion of $\pi_1(X_a)$ and
$\pi_1((X_a)^{\sigma})$ are isomorphic. In the last section we shall give explicit examples
where the actual fundamental groups are not isomorphic.

Another interesting consequence is the following.
Observe that the absolute Galois group
$\Gal(\bar{\Q} /\Q)$ acts on the set  of connected components of the (coarse) moduli spaces of minimal
surfaces of general type.
Theorem \ref{fundamentalgroup} has as a consequence that this action of $\Gal(\bar{\Q} /\Q)$
  {\em does not}  induce an action on the set of isomorphism classes of fundamental groups of
surfaces of general type.

\begin{cor}
$\Gal(\bar{\Q} /\Q)$
  {\em does not}  act on the set of isomorphism classes of fundamental groups of
surfaces of general type.
\end{cor}
\Proof
 In fact, complex conjugation does not change the isomorphism
class of the fundamental group ($X$ and $\bar{X}$ are diffeomorphic). Now, if we had an action on the set
of isomorphism classes of fundamental groups, then the whole  normal closure $\mathfrak H$ of the  $\Z / 2$ generated by
complex conjugation (the set of automorphisms of finite order, by the cited theorem of E. Artin,
see corollary 9.3 in \cite{Lang}) would act trivially. 

By Theorem \ref{fundamentalgroup} the subgroup  $\mathfrak H$ would then be equal to  the union of these
 elements of order $2$ in  $\Gal(\bar{\Q} /\Q)$. But a group where each element has order $\leq 2$
 is abelian, and again we would have a normal abelian subgroup, $\mathfrak H$, of $\Gal(\bar{\Q} /\Q)$,
 contradicting \ref{noAbeliankernel}.
 
\qed

The above arguments show that the set of elements $ \s \in \Gal(\bar{\Q} /\Q)$ such that for each surface of general type
$S$ and $S^{\s}$ have isomorphic fundamental groups is indeed a subgroup where all elements of order two,
in particular it is an abelian group of exponent 2.

\begin{Question} (Conjecture 2.5 in \cite{catbumi})
Is it true that for each $ \s \in \Gal(\bar{\Q} /\Q)$, different from the identity and from complex conjugation,
there exists a  surface of general type $S$ such that
$S$ and $S^{\s}$ have non  isomorphic fundamental groups ?

\end{Question}

It is almost impossible to calculate explicitly the fundamental groups of the
 surfaces constructed above, since one has to explicitly  calculate the monodromy of the Belyi function  of
the very special hyperelliptic curves $C_a$.

Therefore we give in the next section  explicit examples of pairs of   rigid surfaces with
non isomorphic fundamental groups which are Galois conjugate.

\section{Explicit examples}

In this section we provide, as we already mentioned,  explicit examples of pairs of  surfaces with
non isomorphic fundamental groups which are conjugate under the absolute Galois group. Hence 
they have non isomorphic fundamental groups with
isomorphic profinite completions 
(recall that the completion of a group $G$ is the inverse limit
$$ {\hat G} = lim_{K\unlhd_f G}  (G/K),$$
of the factors $G/K$, $K$ being a normal subgroup of finite index in $G$).

The surfaces in our examples are rigid.

In fact, we can prove the following
\begin{teo}\label{beauville}
Beauville surfaces  yield explicit examples of Galois conjugate surfaces with
non-isomorphic fundamental groups (whose profinite completions are isomorphic).

\end{teo}

We consider (see \cite{almeria} for an elementary treatment of what follows) polynomials 
 with only two critical values:
$\{0, 1 \}$.

Let $P \in \mathbb{C}[z]$ be a polynomial with critical values
$\{0,1\}$. 

In order not to have infinitely many polynomials with the same branching behaviour, one considers {\em
normalized polynomials}
 $P(z):= z^n + a_{n-2}z^{n-2} + \ldots a_0$. The condition that $P$ has only $\{0, 1\}$ as critical values,
implies, as we shall briefly recall, that $P$ has coefficients in $\bar{\mathbb{Q}}$. Denote by $K$ the
number field generated by the coefficients of $P$. 

Fix  the types $(m_1, \ldots, m_r)$ and $(n_1, \ldots, n_s)$ of the cycle
decompositions of the respective local monodromies around $0$ and $1$: we can then write our polynomial $P$
in two ways, namely as:
$$
P(z) = \prod_{i=1}^r (z - \beta_i)^{m_i},
$$ and 
$$P(z)  = 1 + \prod_{k=1}^s (z - \gamma_k)^{n_k}.$$

We have the equations $F_1 = \sum m_i \beta_i = 0$ and $F_2=\sum n_k \gamma_k = 0$ (since $P$ is
normalized). Moreover, $m_1 + \ldots + m_r = n_1 + \ldots + n_s = n = degP$ and therefore, since
$\sum_j (m_j-1) + \sum_i(n_i -1) = n-1$, we get $r+s = n+1$.

Since we have
$\prod_{i=1}^r (z -
\beta_i)^{m_i} = 1+ \prod_{k=1}^s (z - \gamma_k)^{n_k}$, comparing coefficients we obtain further
$n-1$ polynomial equations with integer coefficients in the variables $\beta_i$, $\gamma_k$, which we
denote by $F_3=0, \ldots, F_{n+1}=0$. Let $\mathbb{V}(n;(m_1,\ldots,m_n),(n_1,\ldots,n_s))$ be the
algebraic set in affine $(n+1)$-space defined by the equations $F_1=0, \ldots,F_{n+1}=0$. Mapping a
point of this algebraic set to the vector $(a_0,\ldots,a_{n-2})$ of coefficients of the corresponding
polynomial $P$ we obtain a set 
$$\mathbb{W}(n;(m_1,\ldots,m_n),(n_1,\ldots,n_s))$$ (by elimination of variables) in affine $(n-1)$
space. Both these are finite algebraic sets  defined over $\Q$ since by Riemann's existence theorem  they
are either empty or have dimension $0$. 

Observe also that the equivalence classes of monodromies $ \mu \colon \pi_1 (\PP^1 \setminus \{ 0.1.\infty \})
\ra \FS_n$ correspond to the orbits of the group of n-th roots of $1$
(we refer to \cite{almeria} for more details).

\begin{lem}\label{ex} $$\mathbb{W}:=\mathbb{W}(7;(2,2,1,1,1);(3,2,2))$$
 is irreducible over
$\mathbb{Q}$ and splits into two components over $\mathbb{C}$.
\end{lem}

\begin{proof}
This can easily be  calculated  by a MAGMA routine.\end{proof}

 The above lemma implies that
$\Gal(\bar{\mathbb{Q}}/\mathbb{Q})$ acts transitively on
$\mathbb{W}$. Looking at the possible monodromies,  one sees that there are exactly two real non
equivalent polynomials.

In both cases, which will be explicitly described later on, the two
permutations, of types
$(2,2)$ and
$(3,2,2)$, are seen to generate
$\mathfrak{A}_7$ and the respective normal closures of the two polynomial maps 
are easily seen to give (we use here the fact that the automorphism group of $\mathfrak{A}_7$ is
$\mathfrak{S}_7$) nonequivalent  triangle curves $D_1$, $D_2$. 

By Hurwitz's formula, we see that $g(D_i) = \frac{|\mathfrak{A}_7|}{2}(1 - \frac{1}{2} - \frac{1}{6}
-\frac{1}{7}) + 1 = 241$.

\begin{df}
Let $(a_1,a_2,a_3)$ and $(b_1,b_2,b_3)$ be two spherical systems of generators of a finite group $G$ of
the same signature, i.e., $\{ord(a_1), ord(a_2), ord(a_3) \} = \{ord(b_1), ord(b_2), ord(b_3) \}$.
Then $(a_1,a_2,a_3)$ and $(b_1,b_2,b_3)$ are called {\em Hurwitz equivalent} iff they are equivalent
under the equivalence relation generated by
$$
(a_1,a_2,a_3) \equiv (a_2, a_2^{-1}a_1a_2, a_3),
$$
$$
(a_1,a_2,a_3) \equiv (a_1,a_3, a_3^{-1}a_2a_3).
$$
\end{df}

It is well known that two such triangle curves are isomorphic,
compatibly with the action of the group $G$, if and only if the two spherical
systems of generators are {\em Hurwitz equivalent}.  

\begin{lem}
There is exactly one Hurwitz equivalence class of triangle curves given by 
a {\em spherical system of generators of   signature $(5,5,5)$} of
$\mathfrak{A}_7$. 
\end{lem}

\begin{proof}
This is shown by an easy MAGMA routine.
\end{proof}
\begin{oss}
In other words, if $D_1$ and $D_2$ are two triangle curves given by spherical systems
of generators of signature
$(5,5,5)$ of
$\mathfrak{A}_7$, then $D_1$ and $D_2$ are not only isomorphic as algebraic curves, but they have the
same action of $G$.
\end{oss}

Let $D$ be  the triangle curve given by a(ny) spherical systems
of generators of  signature
$(5,5,5)$ of
$\mathfrak{A}_7$. Then by Hurwitz' formula $D$ has genus 505. 

Consider the two triangle curves $D_1$ and $D_2$ as in example \ref{ex}.
Clearly $\mathfrak{A}_7$ acts freely on $D_1 \times D$ as well as on
$D_2 \times D$ and we obtain two non isomorphic Beauville surfaces 
$S_1 : = (D_1 \times D) / G$, $S_2 : = (D_2 \times D) / G$.

\begin{prop} 1) $S_1$ and $S_2$ have nonisomorphic fundamental groups.

2) There is a field automorphism $\sigma \in \Gal(\bar{\Q}/\Q)$ such that $S_2 =
(S_1)^{\sigma}$. In particular, the profinite completions of $\pi_1(S_1)$ and  $\pi_1(S_2)$ are isomorphic.
\end{prop}

\begin{oss}
The above proposition proves theorem \ref{beauville}.
\end{oss}
\begin{proof}
 1) Obviously, the two surfaces $S_1$ and $S_2$ have the same topological Euler characteristic. If they had isomorphic
fundamental groups, by theorem 3.3 of \cite{cat03}, $S_2$ would be the complex conjugate surface of
$S_1$. In particular,  $C_1$ would be the complex conjugate triangle curve
of  $C_2$: but this is absurd since  we shall show that both $C_1$ and
$C_2$ are real triangle curves. 

2) We know that $(S_1)^{\sigma} = ((C_1)^{\sigma} \times (C)^{\sigma}) /G$.  Since there is only one
Hurwitz class of triangle curves given by a spherical system  of generators of  signature
$(5,5,5)$ of $\mathfrak{A}_7$, we have $(C)^{\sigma} \cong C$ (with the same action of $G$).
\end{proof}

We determine now explicitly the respective fundamental groups of $S_1$ and $S_2$.

In general, let $(a_1, \ldots, a_n)$ and $(b_1, \ldots, b_m)$ be two sets of spherical generators of a finite
group $G$ of respective order  signatures $r:=(r_1, \ldots, r_n)$,
$s:=(s_1, \ldots, s_m)$.  We denote the corresponding `polygonal' curves by $D_1$, resp. $D_2$.

Assume now that the diagonal action of $G$ on $D_1 \times D_2$ is free. We get then the smooth surface
$S:= ( D_1 \times D_2 ) /G$, isogenous to a product.

Denote by $T_r := T(r_1, \ldots, r_n)$ the {\em polygonal group} 
$$
\langle x_1, \ldots, x_{n-1} | x_1^ {r_1} =
\ldots = x_{n-1}^ {r_{n-1}} = (x_1 x_2 \ldots  x_{n-1})^ {r_n} = 1\rangle.
$$ 

We have the exact sequence (cf. \cite{isogenous} cor. 4.7)
$$ 1 \rightarrow \pi_1 \times \pi_2 \rightarrow T_r \times T_s
\rightarrow G \times G \rightarrow 1,
$$ where $\pi_i := \pi_1(D_i)$.

Let $\Delta_G$ be the diagonal in $G \times G$ and let $H$ be the inverse image of $\Delta_G$ under
$\Phi \colon T_r \times T_s
\rightarrow G \times G$.  We get the exact sequence 
$$ 1 \rightarrow \pi_1 \times \pi_2 \rightarrow H
\rightarrow G \cong \Delta_G \rightarrow 1.
$$ 

\begin{oss}
$\pi_1(S) \cong H$ (cf. \cite{isogenous} cor. 4.7).
\end{oss}

We choose now an arbitrary spherical system of generators of  signature $(5,5,5)$ of $\mathfrak{A}_7$, 
for instance
$((1,7,6,5,4), (1,3,2,6,7), (2,3,4,5,6))$. Note that we use here MAGMA's notation, where permutations act
on the right (i.e., $ab$ sends $x$ to $(xa)b$).

A MAGMA routine shows that 
\begin{equation}
((1,2)(3,4), (1,5,7)(2,3)(4,6), (1,7,5,2,4,6,3))
\end{equation}
 and 
\begin{equation}
((1,2)(3,4), (1,7,4)(2,5)(3,6), (1,3,6,4,7,2,5))
\end{equation}
are two representatives of spherical generators of  signature $(2,6,7)$ yielding
two non isomorphic  triangle curves $C_1$ and $C_2$, each of which
is isomorphic to its complex conjugate. In fact, an alternative
 direct argument is as follows.  First of all, $C_i$ is isomorphic to its complex
conjugate triangle curve since, for an appropriate choice of the real base point,
complex conjugation sends $ a \mapsto a^{-1}, b \mapsto b^{-1}$
and  one sees that the two corresponding monodromies
are permutation equivalent (see  Figure \ref{figura1} and  Figure \ref{figura2}).

Moreover, since $\Aut(\mathfrak{A}_7) = 
\mathfrak{S}_7$, if the two triangle curves were isomorphic, then the two monodromies were conjugate
in $\mathfrak{S}_7$. That this is not the case is seen again by the following pictures. 
\begin{figure}[htbp]
\begin{center}
\input{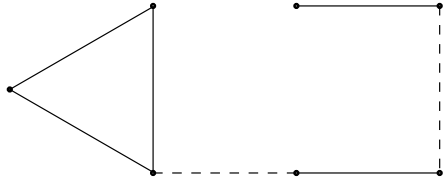}
\end{center}
\caption{Monodromy corresponding to (1)}
\label{figura1}
\end{figure}

\begin{figure}[htbp]
\begin{center}
\input{figura1.pstex_t}
\end{center}
\caption{Monodromy corresponding to (2)}
\label{figura2}
\end{figure}

The two corresponding
homomorphisms
$\Phi_1 \colon  T_{(2,6,7)} \times T_{(5,5,5)}
\rightarrow \mathfrak{A}_7 \times \mathfrak{A}_7$ and $\Phi_2 \colon  T_{(2,6,7)} \times T_{(5,5,5)}
\rightarrow \mathfrak{A}_7 \times \mathfrak{A}_7$ give two exact sequences

$$ 1 \rightarrow \pi_1(C_1) \times \pi_1(C) \rightarrow T_{(2,6,7)} \times T_{(5,5,5)}
\rightarrow \mathfrak{A}_7 \times \mathfrak{A}_7 \rightarrow 1,
$$ and
$$ 1 \rightarrow \pi_1(C_2) \times \pi_1(C) \rightarrow T_{(2,6,7)} \times T_{(5,5,5)}
\rightarrow \mathfrak{A}_7 \times \mathfrak{A}_7 \rightarrow 1,
$$

yielding two non isomorphic fundamental groups $\pi_1(S_1) = \Phi_1 ^{-1} (\Delta_{\mathfrak{A}_7})$
and $\pi_1(S_2) = \Phi_2 ^{-1} (\Delta_{\mathfrak{A}_7})$ fitting both in an exact sequence of type
$$ 1 \rightarrow \pi_{241} \times \pi_{505}  \rightarrow \pi_1(S_j) 
\rightarrow
\Delta_{\mathfrak{A}_7} \cong \mathfrak{A}_7 \rightarrow 1,
$$

where $\pi_{241} \cong \pi_1(C_1) \cong \pi_1(C_2)$, $\pi_{505} = \pi_1(C)$.

Using the same trick that we used for our main theorems, namely, using a surjection
of a group $\Pi_g \ra \mathfrak{A}_7$, $g \geq 2$, we get infinitely many examples of pairs of
fundamental groups which are nonisomorphic, but which have 
isomorphic profinite completions. 

This implies the following

\begin{teo}
There is an infinite sequence $g_1 < g_2 < \ldots < g_i < \ldots$ and for each $g_i$ there is a pair of surfaces $S_1(g_i)$ and $S_2(g_i)$ isogenous to a product, such that 
\begin{itemize}
\item the corresponding connected components $\mathfrak{N}(S_1(g_i))$ and $\mathfrak{N}(S_2(g_i))$ are disjoint,
\item there is a $\sigma \in \Gal(\bar{\Q} / \Q)$ such that $\sigma(\mathfrak{N}(S_1(g_i))) = \mathfrak{N}(S_2(g_i))$,
\item $\pi_1(S_1(g_i))$ is non isomorphic to $\pi_1(S_2(g_i))$, but they have isomorphic profinite completions,
\item the fundamental groups fit into an exact sequence $$ 1 \rightarrow \Pi_{241} \times \Pi_{g_i'}  \rightarrow \pi_1(S_j(g_i)) 
\rightarrow
 \mathfrak{A}_7 \rightarrow 1, \ j = 1,2.
$$
\end{itemize}
\end{teo}
\begin{oss}
1) Many more explicit examples as the one above (but with cokernel group different from 
${\mathfrak{A}_7}$) can be obtained
using polynomials with two critical values.

2) A construction of polynomials with two critical values having a
very large Galois orbit was proposed to us by D. van Straten.
\end{oss}

{\bf Acknowledgements.} The research of the authors was performed in the realm of the
Forschergruppe  790 `Classification of algebraic surfaces and compact complex manifolds' of the D.F.G.
(Deutsche Forschungs Gemeinschaft). 
The first two authors are grateful to the KIAS Seoul
for hospitality in August 2012, where the second author was a KIAS scholar, and the final version of the paper was begun.

They mourn and miss their friend and collaborator Fritz Grunewald, who passed away 
on  March 21,  2010.

Thanks to Ravi Vakil for his interest in our work and for pointing
out a minor error in a very first version of this note. 

Last, but not least, thanks to Zoe Chatzidakis, Minhyong Kim and Umberto Zannier for informing us
and  providing references
for proposition \ref{noAbeliankernel} during the ERC activity at the Centro De Giorgi, Pisa, in October 2012.

\noindent {\bf Authors' address:}

\noindent Ingrid Bauer, Fabrizio Catanese\\ 
Lehrstuhl Mathematik VIII, \\
Mathematisches Institut der Universit\"at Bayreuth\\
  D-95440 Bayreuth, Germany\\
\noindent (email: Ingrid.Bauer@uni-bayreuth.de, Fabrizio.Catanese@uni-bayreuth.de)\\

\end{document}